\documentclass[11pt]{article}

\usepackage{girth}
\usepackage{scalerel}

\newcommand{\AvgDeg}{\mathrm{AvgDegree}}
\newcommand{\VExp}{\Psi}
\newcommand{\BikeConst}{c_{\scaleto{\ref{lem:bike-freeness-random-graph}}{5pt}}}
\newcommand{\BikeConstexp}{c_{\scaleto{\ref{lem:bike-freeness-random-graph}}{3pt}}}
\newcommand{\HYConst}{C_{\mathrm{HY}}}
\newcommand{\GWConst}{c_{\scaleto{\ref{lem:gauswave}}{5pt}}}

\newcommand{\Cmain}{C_{ \scaleto{\ref{thm:main}}{5pt}}}

\begin{document}
\title{Many Nodal Domains in Random Regular Graphs}
\author{Shirshendu Ganguly\thanks{\texttt{sganguly@berkeley.edu}. Supported by NSF Grant DMS-1855688, NSF Career Grant DMS-1945172, and a Sloan Fellowship. } \\ UC Berkeley \and Theo McKenzie\thanks{\texttt{mckenzie@math.berkeley.edu}. Supported by NSF GRFP Grant DGE-1752814. } \\ UC Berkeley\and Sidhanth Mohanty\thanks{\texttt{sidhanthm@cs.berkeley.edu}. Supported by Google PhD Fellowship.}\\ UC Berkeley \and Nikhil Srivastava\thanks{\texttt{nikhil@math.berkeley.edu}. Supported by NSF Grant CCF-2009011.} \\ UC Berkeley}

\maketitle

\begin{abstract}
    Let $G$ be a random $d$-regular graph.  We prove that for every constant $\alpha > 0$, with high probability every eigenvector of the adjacency matrix of $G$ with eigenvalue less than $-2\sqrt{d-2}-\alpha$ has $\Omega(n/\polylog(n))$ nodal domains.
\end{abstract}

\section{Introduction}
Courant's  nodal domain theorem states that the zero set of the $k$th smallest Dirichlet eigenfunction of the Laplacian on a smooth bounded domain in $\R^d$ partitions it into at most $k$ connected components \cite{courant2008methods}. These components, known as the {\em nodal domains} of the eigenfunction, have garnered significant interest over time in spectral geometry and mathematical physics (see e.g. \cite{zelditch2017eigenfunctions}). The analogous definition for a finite discrete graph $G=(V,E)$ is the following.

\begin{definition}[Nodal domains]
    A (\emph{weak}) \emph{nodal domain} of a function $f:V\rightarrow\R$ on $G$ is a maximal connected subgraph $S$ of $G$ such that $f(u)\geq 0$ for all $u\in S$ or $f(u)\leq 0$ for all $u\in S$.  A \emph{strong nodal domain} of $f:V\rightarrow \R$ on $G$ is a maximal connected subgraph $S$ of $G$ such that $f(u)> 0$ for all $u\in S$ or $f(u)<0$ for all $u\in S$. 

\end{definition}
Fiedler \cite{fiedler1975property} showed that for a tree, the eigenvector of the $k$th smallest eigenvalue of the discrete Laplacian (defined as $L_G=D_G-A_G$ where $D_G$ is the diagonal matrix of vertex degrees and $A_G$ is the adjacency matrix) has exactly $k$ nodal domains. Davies, Gladwell, Leydold, and Stadler \cite{davies2000discrete} showed that for an arbitrary graph that the $k$th Laplacian eigenvector has at most $k$ nodal domains and at most $k+m-1$ strong nodal domains, where $m$ is the multiplicity of the $k$th eigenvalue.  Berkolaiko \cite{berkolaiko2008lower} showed that for a connected graph with $n$ vertices and $n+\ell-1 $ edges (such that removing $\ell$ edges would produce a tree) the $k$th eigenvector of a Schr\"odinger operator with arbitrary potential has between $k-\ell$ and $k$ nodal domains. Beyond these results, we are not aware of any lower bounds on the number of nodal domains of eigenvectors of any large class of graphs.

 Our main result is the following lower bound on the number of nodal domains of a random regular graph\footnote{We restrict our attention to weak nodal domains as there are at least as many strong domains as weak domains.}.  We refer to a nodal domain with a single vertex as a {\em singleton} nodal domain. 
\begin{theorem}\label{thm:main}
   Fix $d\geq 3$ and $\alpha>0$ and let $G$ be a random $d$-regular graph on $n$ vertices. Then with probability $1-o(1)$ as $n\rightarrow\infty$, every eigenvector of $A_{G}$ with eigenvalue $\lambda\leq -2\sqrt{d-2}-\alpha$ has $\Omega\left(\frac{n}{\log^{C_{ \scaleto{\ref{thm:main}}{3pt}}}(n)}\right)$ singleton nodal domains,
   where $\Cmain\le 301$ is an absolute constant.
\end{theorem}

 Note that for large enough $n$, almost every $d-$regular graph  has at least $\Omega(d^{-3/2}n)$ eigenvalues with $\lambda\leq -2\sqrt{d-2}$, as the spectrum of $A_{G}$ converges weakly to the Kesten-McKay measure \cite{mckay1981expected}. Since the Laplacian of a $d-$regular graph is  equal to $dI-A_{G}$, the conclusion of the theorem also holds for the ``high energy'' eigenvectors of the Laplacian with eigenvalues $\lambda\ge  d+2\sqrt{d-2}+\alpha$;  we will accordingly also refer to highly negative eigenvalues of the adjacency matrix as high energy. 

The proof of \pref{thm:main} appears in \pref{sec:either} to \pref{sec:many} and employs tools from random matrix theory ($\ell_\infty$ delocalization of eigenvectors of random regular graphs \cite{huang2021spectrum}), graph limits (weak convergence of eigenvectors of random regular graphs \cite{backhausz2019almost}), and combinatorics (expansion and short cycle counts of random regular graphs), and is outlined in \pref{sec:sketch}. The conceptual phenomenon articulated by the proof is that (under certain conditions) high energy eigenvectors of graphs cannot simultaneously have few nodal domains and be delocalized. A simple demonstration of this tension for the easier case of $d=3,4$ is presented in \pref{sec:warmup}.
Due to the use of a weak convergence argument, there is no effective bound on the $o(1)$ probability in the statement of \pref{thm:main}, and the proof requires $d$ to be constant.

We complement \pref{thm:main} by observing in \pref{sec:xml} (\pref{thm:xml}) that by an application of the expander mixing lemma, {\em every} non-leading eigenvector $f$ of a $d$-regular expander graph $G$ with sufficiently large spectral gap has two nodal domains which together contain a constant fraction of the vertices of $G$.  


\subsection{History and Related Work}
\noindent{\em Random Graphs.} 
 Dekel, Lee, and Linial \cite{dekel2011eigenvectors} initiated the study of nodal domains of eigenvectors of Erd\"os-R\'enyi $G(n,p)$ random graphs. They showed that for constant $p$, with high probability all but $O(1)$ of the vertices are contained in two large nodal domains for every non-leading adjacency eigenvector. Arora and Bhaskara \cite{arora2011eigenvectors} improved this by establishing that when $p\geq n^{-1/19+o(1)}$ there are typically exactly $2$ nodal domains in each non-leading eigenvector. H. Huang and Rudelson \cite{huang2020size} proved that these two domains are approximately the same size for eigenvectors of eigenvalues macroscopically away from the edge when $p \in [n^{-c},1/2]$ for some fixed $c$ and also for the first and last $e^{c(\log\log n)^2}$ eigenvectors when $p\in (0,1)$ is constant. Linial suggested studying the shape of these nodal domains; for example, how many vertices are on the boundary of a domain, what is the distribution of distances to the boundary, etc. For sufficiently dense graphs sampled from $G(n,p)$, this geometry turned out to be trivial --- in particular, Rudelson \cite[Section 5.2]{rudelson2017delocalization} showed that with high probability, for $G(n,p)$ with fixed $p\in (n^{-c},1)$, every vertex is adjacent to $\Omega(n/\polylog n)$ vertices that have the opposite sign in each eigenvector $f$. This left open the question of nontrivial structure of the nodal domains for sparse graphs\footnote{As a starting point, Eldan, H. Huang, and Rudelson asked in 2020 \cite{rudelsonsimons} whether the most negative eigenvector of a sparse $G(n,p)$ graph has more than two nodal domains. }.   \pref{thm:main} and \pref{thm:xml} show that both the number and the geometry of nodal domains is nontrivial for high energy eigenvalues of sparse random regular graphs.
 
In contrast to the situation for dense graphs, Dekel, Lee, and Linial observed that in simulations, a randomly selected $d$-regular graph with $d$ constant has a number of nodal domains that increases as the eigenvalue becomes more negative. Our results confirm their observation that the most negative eigenvalues have many nodal domains.\\

\noindent {\em Random Matrix Theory and Graph Limits.} The results for $G(n,p)$ described above rely crucialy on delocalization estimates in random matrix theory. There are two relevant notions of delocalization, namely $\ell_\infty$ norm bounds (the strongest being of order $\log^C n/\sqrt{n}$) and ($\ell_2$) {\em no-gaps delocalization}, which asserts that every subset of $tn$ vertices has at least a $\beta(t)$ fraction of the $\ell_2$ mass of an  eigenvector. Generally speaking, $\ell_\infty$ bounds  are derived via delicate Green's functions estimates \cite{bauerschmidt2019local} whereas no-gaps bounds are derived via geometric arguments \cite{rudelson2016no}. No-gaps delocalization is so far only known for sufficiently dense graphs, and remains open for sparse random graphs.

The proof of \pref{thm:main} relies on both notions of delocalization and combines them in a new way. We first consider no-gaps delocalization at scale $t=1-\delta$ for a small constant $\delta$; if this property holds for an eigenvector,  we employ a weak convergence result of Backhausz and Szegedy \cite{backhausz2019almost} to argue that the local distribution of eigenvector entries around a randomly chosen vertex behaves like a Gaussian wave (defined in Section \ref{sec:gwave}), implying that a random vertex is a singleton nodal domain with constant probability. Otherwise, we apply the $\ell_\infty$ delocalization estimate of \cite{bauerschmidt2019local,huang2021spectrum} to the subset of $\delta n$ vertices on which the eigenvector is $\ell_2$-localized; the $\ell_\infty$ bound allows us to simplify and exploit the locally almost-treelike structure of the graph on this subset and deduce many singleton nodal domains via a different argument which hinges on the negativity of the eigenvalue $\lambda$. Thus, we sidestep the current lack of no-gaps estimates for random regular graphs, as well as the difficulty of examining individual eigenvector entries solely using the Green's function method\footnote{The Green's function $(A-zI)^{-1}$ of a random regular graph can only approximate that of the infinite tree when $\Im(z)\geq \polylog n/n$, meaning that it inherently reflects the aggregate behavior of $\polylog n$ eigenvectors.}. \\

\noindent {\em Mathematical Physics.} The field of quantum chaos aims to relate the classical dynamics of the geodesic flow on a manifold to the behavior of its high energy Laplacian eigenfunctions \cite{rudnick2008quantum}, and the number of nodal domains has also been studied in this context \cite{blum2002nodal}. A guiding question in this area is Berry's random wave conjecture \cite{berry1977regular}, which asserts that the high energy eigenfunctions of quantum chaotic billiards behave like ``Gaussian random waves'' in the limit. Random $d$-regular graphs have studied as a discrete model of quantum chaos \cite{kottos1997quantum,band2007nodal,smilansky2013discrete}; in particular, a discrete analogue of Berry's conjecture considered in \cite{elon2008eigenvectors} asserts that the bulk eigenvectors of random $d$-regular graphs have a (locally) jointly Gaussian distribution with a specific {\em nonzero} covariance matrix depending on the degree $d$. This conjecture implies the existence of many nodal domains in random regular graphs. \pref{thm:main} proves the implication of the conjecture for sufficiently negative $\lambda$, and one branch of its proof (\pref{sec:either}) is directly inspired by the ``Gaussian wave'' heuristic, which we make rigorous via the weak convergence result of \cite{backhausz2019almost}.

 \subsection{Low degree case}\label{sec:warmup}
 As a warm-up, we prove a weaker version of \pref{thm:main} which applies to any eigenvector of a regular graph with sufficiently negative eigenvalue and an $\ell_\infty$ bound.
 
 \begin{prop}\label{prop:warmup}
 For $\alpha,\eta>0$, $d\geq 3$, assume $f$ is an eigenvector of a $d$-regular graph $G=(V,E)$ with eigenvalue $\lambda\leq -(d-1)-\alpha$ and
 \begin{equation}\label{eq:infbound}
     \|f\|_\infty\leq \frac{\eta}{\sqrt n}.
 \end{equation}
Then $f$ has at least
\[\frac{n}{(2\eta)^{2+\frac{\log(d-1)}{\log(1+\alpha/(d-1))}}}\]
nodal domains. 
 \end{prop}

 \begin{proof}
Assume that $u\in V$ is not a singleton nodal domain and $|f(u)|\geq \frac1{2\sqrt n}$. Then $u$ has at most $d-1$ neighbors $v$ such that $f(u)f(v)\leq 0$, so as $\sum_{v\sim u}f(v)=\lambda f(u)$, we must have that for some neighbor $v$ of $u$, $|f(v)|\geq (1+\alpha/(d-1))|f(u)|$. Repeating this argument, if there are no singleton nodal domains at distance at most $k$ from $u$, then there is a path $(u=x_0,\ldots, x_k)$ such that $|f(x_i)|\geq (1+\alpha/(d-1))|f(x_{i-1})|$ for each $i$. By \eqref{eq:infbound}, we must have $k\leq \tilde k$ for 
 \[
 \tilde k:=\frac{\log(2\eta)}{\log (1+\frac{\alpha}{d-1})}.
 \]
 
 Every $u$ with $|f(u)|\geq\frac{1}{2\sqrt n}$ must have a vertex $w$ that is a singleton nodal domain and $d(u,w)\leq \tilde k$. By \eqref{eq:infbound}, there are at least $\frac{3}{4}n/\eta^2$ vertices $u$ with $|f(u)|\geq 1/2\sqrt n$. 
 
 Any vertex $w$ has at most $d(d-1)^{\tilde k-1}$ vertices at distance at most $\tilde k$. Therefore there are at least
 \[
    \frac{\frac{3}{4}\cdot\frac{n}{\eta^2}}{d(d-1)^{\tilde k-1}}\ge\frac{n}{(2\eta)^{2+\frac{\log(d-1)}{\log(1+\alpha/(d-1))}}}
 \]
 singleton nodal domains.
 \end{proof}

The $\ell_\infty$ delocalization bound of \cite{huang2021spectrum} corresponds to $\eta=\polylog n$. Thus if $d\le 4,\alpha>0$ are fixed and $\lambda\le -(d-1)-\alpha$, \pref{prop:warmup} yields $\Omega(n/\polylog n)$ nodal domains for an eigenvector of a random $d$-regular graph, recovering the conclusion of \pref{thm:main} up to polylogarithmic factors in the spectral window $[-d,-(d-1)-\alpha]$. We recall that every nontrivial eigenvalue $\lambda$ of a random $d-$regular graph satisfies $|\lambda|\leq 2\sqrt{d-1}+o(1)$ with high probability \cite{friedman2003proof}, so for $d>5$ there are typically no eigenvectors with $\lambda\le -(d-1)$ and \pref{prop:warmup} is vacuous. To improve the required bound on $\lambda$ from $-(d-1)$ to $-2\sqrt{d-2}$, we shift from a local analysis of the entries of $f$ to a more global one.

\subsection{Proof outline and organization}\label{sec:sketch} 

In \pref{sec:prelim}, we go over notation and some preliminary statements. In \pref{sec:either}, we use the weak convergence result of Backhausz and Szegedy \cite{backhausz2019almost} to show that with high probability, if the $\ell_2$ mass of an eigenvector $f$ is not concentrated on a set of size $\delta n$ for a small constant $\delta$, then it has many singleton nodal domains. The remainder of the proof focuses on the case where the eigenvector $f$ is $\ell_2$-localized on a small set $S\subset G$. In \pref{sec:specrad} we give a deterministic upper bound of the spectral radius of ``almost treelike'' graphs in terms of their maximum degree, average degree, and girth; in particular, the bound implies that certain small subgraphs of $G$ have small spectral radius, with high probability. In \pref{sec:localized} we show that if $f$ has few singleton nodal domains in $S$, then we may pass to an edge subgraph $H\subset G[S]$ (of the induced subgraph $G[S]$) of maximum degree at most $d-1$ such that the restriction of $f$ to $S$, denoted by $f_S$, satisfies
\begin{equation}\label{eqn:hlarge}f_S^TA_Hf_S\approx f^TA_G f = \lambda.\end{equation} This is the step in which both the $\ell_2$-localization assumption and the $\ell_\infty$ bound of \cite{huang2021spectrum} are crucially used. If $\lambda$ is sufficiently negative, \eqref{eqn:hlarge} violates the spectral radius bound of \pref{sec:specrad} applied to $H$, so we conclude that there must be many singleton nodal domains of $f$ in $S$. We combine the above  cases to prove \pref{thm:main} in \pref{sec:many}. We conclude by showing that any sparse expander graph contains two nodal domains whose total size is large in \pref{sec:xml}.

\section{Preliminaries}\label{sec:prelim}
\subsection{Notation and basic definitions}
Given a graph $G$ on $n$ vertices, we shall use $V(G)$ to denote its vertex set, $E(G)$ to denote its edge set, and $A_G$ to denote its adjacency matrix.  We will order the $n$ eigenvalues of $A_G$ and denote them as:
\[
    \lambda_{\max}(G)=\lambda_1(G) \ge \lambda_2(G) \ge \dots \ge \lambda_n(G).
\]
For a subset of vertices $S\subseteq V(G)$ we use $G[S]$ to denote the induced subgraph of $G$ on $S$.  We use $N(S)$ to denote the set of vertices that have a neighbor in $S$.  We use $E(S,T)$ to denote the collection of edges with one endpoint in $S$ and one endpoint in $T$.  We use $\ol{S}$ to denote the the set of vertices $V(G)\setminus S$.  We use $B_G(S,\ell)$ to denote the induced subgraph on the set of all vertices of distance at most $\ell$ from $S$, and we write $B_G(v,\ell):=B_G(\{v\},\ell)$.

Given a vector $f\in\R^{V(G)}$, we use $f_S$ to denote the vector in $\R^S$ obtained by restricting $f$ to coordinates in $S$. We also will write $\|f\|:=\|f\|_2$.  For a matrix $A$, we use $\|A\|$ to denote the spectral norm of $A$.

We write $\Pr_{x\sim \mu}(E)$ to denote the probability that a random variable $x$ sampled from the distribution $\mu$ satisfies $E$.

\subsection{Graph theory}
We use the following standard facts about expansion and cycle counts in random regular graphs.
\begin{definition}
    The \emph{spectral expansion} of a graph $G$, denoted $\lambda(G)$, is defined as $\max\{\lambda_2(G),-\lambda_n(G)\}$.
\end{definition}

\begin{definition}
    The \emph{$\eps$-edge expansion} of a graph $G$, denoted $\Psi_{\eps}(G)$, is defined as:
    \[
        \Psi_{\eps}(G)\coloneqq \max_{\substack{S\subseteq V(G) \\ |S|\le \eps n}} \frac{|E(S,\ol{S})|}{|S|}.
    \]
\end{definition}

\begin{definition}[Bicycle-freeness]
    We say $G$ is $\ell$-bicycle-free if for every vertex $v$, $B_G(v,\ell)$ contains at most $1$ cycle.
\end{definition}

\begin{lemma}[Expander Mixing Lemma]    \label{lem:exp-mix-lem}
    Let $G$ be a $d$-regular graph, $S,T\subseteq V(G)$, and $e(S,T)$ is the number of tuples $(u,v)$ such that $u\in S, v\in T$ and $\{u,v\}\in E(G)$.  Then:
    \[
        e(S,T) \in \frac{d}{n}|S|\cdot|T| \pm \lambda(G)\sqrt{|S|\cdot|T|\cdot\left(1-\frac{|S|}{n}\right)\cdot\left(1-\frac{|T|}{n}\right)}.
    \]
\end{lemma}

\begin{lemma}[{Edge expansion in random graphs \cite[Theorem 4.16]{HLW}}]   \label{lem:lossless-edge-expansion}
    Let $G$ be a random $d$-regular graph.  For every $\delta>0$, there is an $\eps > 0$ such that:
    \[
        \VExp_{\eps}(G) \ge d-2-\delta.
    \]
\end{lemma}

\begin{lemma}[{Bicycle-freeness in random regular graphs \cite[Lemma 9]{Bor}}]  \label{lem:bike-freeness-random-graph}
    Let $G$ be a random $d$-regular graph.  There exists an absolute constant $\BikeConst \in(0,1)$ such that with probability $1-o(1)$, $G$ is $\ell$-bicycle-free for any $\ell \le \BikeConst \log_{d-1}n$.
\end{lemma}

We write $G\backslash F$ to signify $(V,E\backslash F)$. We use \pref{lem:bike-freeness-random-graph} to derive the following:
\begin{lemma}   \label{lem:excess-to-girth}
    Let $G$ be a random $d$-regular graph.  Then with probability $1-o_n(1)$ there exists a collection of edges $F$ with cardinality bounded by $(d-1)n^{1-\BikeConstexp/2}$ such that $G\setminus F$ has girth $\ell\coloneqq\frac{\BikeConst}{2}\log_{d-1} n$.
\end{lemma}
\begin{proof}
    Let $\calC$ be the collection of all cycles in $G$ of length at most $\ell$.
    By \pref{lem:bike-freeness-random-graph}, $G$ is $2\ell$-bicycle-free.  Consequently, the collection of graphs given by $\calC'\coloneqq\{B_{G}(C,\ell):C\in\calC\}$ must be pairwise vertex-disjoint.  Indeed, if there are distinct $C,C'\in\calC$ for which $B_{G}(C,\ell)$ and $B_{G}(C',\ell)$ share a vertex $v$, then $B_{G}(v,2\ell)$ contains both $C$ and $C'$ contradicting bicycle-freeness.

    By bicycle-freeness, for any $C\in\calC$, the number of vertices in $B_{G}(C,\ell)$ is at least $(d-1)^{\ell-1} = \frac{n^{\BikeConstexp/2}}{d-1}$, and by vertex-disjointness of the balls around cycles, $|\calC'|\le(d-1)n^{1-\BikeConstexp/2}$.  However, since $|\calC|=|\calC'|$, we have the same bound on $|\calC|$.  We can then construct $F$ by choosing one edge per $C\in\calC$, which completes the proof.
\end{proof}

\subsection{Delocalization of eigenvectors of random regular graphs}
A key ingredient in our proof is the following result about $\ell_{\infty}$-delocalization of eigenvectors in random regular graphs, as stated in \cite[Theorem 1.4]{huang2021spectrum}, which built on the previous result \cite{bauerschmidt2019local}.
\begin{theorem} \label{thm:eigvec-deloc}
    Let $d\ge 3$ be a constant, and let $G$ be a random $d$-regular graph.  With probability $1-O(n^{-1+o(1)})$ for all eigenvectors $v$:
    \[
        \|v\|_{\infty} \le \frac{\log^{\HYConst} n}{\sqrt{n}}\|v\|,
    \]
    where $\HYConst\le 150$ is an absolute constant independent of $d$.
\end{theorem}

\subsection{Gaussian wave}\label{sec:gwave}

Our results also use results concerning the before-mentioned Gaussian wave.
\begin{definition}
    Consider the infinite $d$-regular tree $T_d$ with vertex set $V_d$ and origin $o$. An \emph{eigenvector process} with eigenvalue $\lambda$ is a joint distribution $\{X_v\}_{v\in V_d}$, such that it is invariant under all automorphisms of the tree, $\E(X_o^2)=1$, and satisfies the eigenvector equation
\begin{equation}\label{eq:eigenvector}
\lambda X_o=\sum_{v\sim o} X_v
\end{equation}
with probability 1. 
\end{definition}
Observe that the eigenvector process must satisfy the eigenvector equation at every vertex by automorphism invariance, and that by taking the expectation of \eqref{eq:eigenvector} and automorphism invariance, if $\E(X_o)\neq 0$, then $\lambda=d$. 
\begin{definition}
A \emph{Gaussian wave} is an eigenvector process that is also a Gaussian process.
\end{definition}
\begin{theorem}[Theorem 1.1 of \cite{elon2009gaussian}]
For any $-d\leq \lambda \leq d$, there exists a unique Gaussian wave with parameter $\lambda$. 
\end{theorem}

We call this Gaussian wave $\Lambda_\lambda$.
\begin{definition}
The \emph{L\'evy Prokhorov distance} between two Borel probability measures $\mu_1$ and $\mu_2$ on $\R^k$ is given by 
\begin{equation*}
\tilde d(\mu_1,\mu_2):=\inf\{\epsilon>0|\forall A\in \calB_k,\mu_1(A)\leq \mu_2(A_\epsilon)+\epsilon
\textnormal{ and } \mu_2(A)\leq \mu_1(A_\epsilon)+\epsilon\},
\end{equation*}
where $\calB_k$ is the set of Borel measurable sets in $\R^k$ and $A_\epsilon$ is the neighborhood of radius $\epsilon$ around $A$.
\end{definition}

Define $C_\ell$ to be the number of vertices in $B_{T^d}(v,\ell)$, where $T_d$ is the infinite $d$-regular tree, and $v$ is an arbitrary vertex. Namely \[C_\ell:=
1+\frac{d((d-1)^{\ell}-1)}{d-2}.
\]
A vector $f\in \R^{V(G)}$ on the vertices of a graph $G$ on $n$ vertices defines the following distribution $\nu_{G,f,\ell}$ on $\R^{C_\ell}$. Select a vertex $u\in V$ uniformly at random. Order the vertices in $B(u,\ell)$ by starting a breadth first search at $u$, breaking ties in the order of the search uniformly at random. Create the vector $x:=(x_1,\ldots x_{C_\ell})$ such that $x_k:=\sqrt nf(u_k)$, where $u_k$ is the $k$th vertex in this breadth first search. If $B(u,\ell)$ has fewer than $C_\ell$ vertices, then have $x_k=0$ for $1\leq k\leq C_\ell$. Finally, let $\nu_{G,f,\ell}$ be the distribution of $\overline x(u)$.

\begin{theorem}[Theorem 2 of \cite{backhausz2019almost}]\label{thm:gaussianwave}
 For every $d\geq 3$, $\epsilon>0$ and $R\in \N$, there exists $N$ such that for $n>N$, with probability at least $1-\epsilon$, a random regular graph of degree $d$ on $n$ vertices has the following property. Any eigenvector $f$ of $G$ is such that $\nu_{G,f,R}$ is at most $\epsilon$ in L\'evy-Prokhorov distance from the distribution of $\sigma\cdot \Lambda_\lambda$ restricted to the vertices of $B_{T^d}(o,R)$ for some $\sigma\in[0,1]$, where $\lambda$ is the eigenvalue of $f$.
\end{theorem}

In fact, \cite{backhausz2019almost} proves that there is an $N$ and a $\delta>0$ such that a $G(n,d)$ graph has this property for all normalized vectors $f$ such that there exists a constant $\lambda$ such that $\|(A-\lambda I)f\|\leq \delta$. Namely, this statement is true for all ``pseudo-eigenvectors''.

\section{Either $\ell_2$-localization or many nodal domains}\label{sec:either}
In this section, we show (\pref{lem:many-isolated-or-localized}) that if an eigenvector of a random regular graph is appropriately delocalized in $\ell_2$, then its proximity to the Gaussian wave implies it has many nodal domains. We begin by showing that the root vertex in a Gaussian wave with negative parameter $\lambda$ has a constant probability of being a singleton domain.

\begin{lemma}\label{lem:gauswave}
For $d\geq 3$ and $0<\alpha\leq d$, let
\[
\GWConst:=\frac{\alpha^d}{3^{d+2}d^{d+1}}.
\]

Assume that $\lambda\leq-\alpha$. With probability at least $\GWConst$, $\{o\}$ is a
singleton nodal domain in $\Lambda_\lambda$ with all entries in $B(o,1)$ of modulus at least $\alpha/5d$.
\end{lemma}
\begin{proof}
The proof proceeds by using the covariance of the Gaussian wave to pass to a Gaussian vector with i.i.d. entries, then showing that with probability at least $\GWConst$, this vector has a direction and norm that imply \pref{lem:gauswave}.

The distribution of $\Lambda_\lambda$ restricted to $B(o,1)$ is given by the multivariate normal distribution $N(\textbf{0},\Sigma)$ for a $(d+1)\times (d+1)$ covariance matrix $\Sigma$. The distribution according to $N(\textbf0,\Sigma)$ is the same as the distribution of $\Sigma^{1/2}g$, where $g$ is a length $(d+1)$ vector with i.i.d. Gaussian $N(0,1)$ entries.  Denote by $\{v_1,\ldots,v_d\}$ the neighbors of $o$ and denote by $e_v$ the elementary vector on $v$. Notice that 
$\langle\Sigma^{1/2}e_o,\Sigma^{1/2}e_{o}\rangle=\E(X_o^2)=1$, and by the eigenvector equation and automorphism invariance $\langle\Sigma^{1/2}e_o,\Sigma^{1/2}e_{v_i}\rangle=\E(X_oX_{v_i})=\lambda/d\leq-\alpha/d$.

Let $\tilde g:=g/\|g\|$. Next, we show that if $\tilde g$ is sufficiently close to $\Sigma^{1/2}e_{o}$, then it must have negative inner product with $\Sigma^{1/2} e_{v_i}$ for each $1\leq i\leq d$. 

If $\langle \tilde g,\Sigma^{1/2}e_o\rangle \geq 1-\frac{\alpha^2}{16d^2}$,

\begin{eqnarray*}
\langle \tilde g,\Sigma^{1/2}e_{v_i}\rangle&=&1-\frac12\|\tilde g-\Sigma^{1/2}e_{v_i}\|^2\\
&\leq& 1-\frac12\left(\|\Sigma^{1/2}e_o-\Sigma^{1/2}e_{v_i}\|-\|\tilde g-\Sigma^{1/2}e_o\|\right)^2\\
&\leq& 1-\left(\sqrt{1-\langle\Sigma^{1/2}e_o,\Sigma^{1/2}e_{v_i}\rangle}-\sqrt{1-\langle\Sigma^{1/2}e_o,\tilde g\rangle}\right)^2\\
&\leq&1-\left(\sqrt{1+\frac\alpha d}-\sqrt{\frac{\alpha^2}{16d^2}}\right)^2\\
&\leq&-\frac{\alpha}{d}-\frac{\alpha^2}{16d^2}+\frac{\alpha}{2d}\sqrt{1+\frac{\alpha}{d}}\\
&\leq &-\frac{\alpha}{5d}
\end{eqnarray*}
 for each $i$. The first inequality is the triangle inequality. The second is the parallelogram law. The last inequality is true as $\alpha/d\leq 1$.

The probability that $\|g\|\geq 1$ is at least the probability that the first coordinate of $g$ has modulus at least 1. As this coordinate is standard normal, this probability is at least 0.3. The probability that $\langle \tilde g,e_o\rangle\geq 1-\frac{\alpha^2}{16d^2}$ is the surface area of the spherical cap where this inequality is true divided by the surface area of the sphere. The surface area of the spherical cap is at least the volume of the $d$ dimensional sphere base of the spherical cap. The radius of the $d$-dimensional sphere is  \[\sqrt{1-\left(1-\frac{\alpha^2}{16d^2}\right)^2}=\sqrt{\frac{\alpha^2}{8d^2}-\frac{\alpha^4}{256d^4}}\geq \frac{\alpha}{3d},\] meaning that the probability that $\langle \tilde g,e_o\rangle \geq 1-\frac{\alpha^2}{16d^2}$ is at least
\[
\left.\left(\left(\frac{\alpha}{3d}\right)^d\cdot \frac{\pi^{d/2}}{\Gamma(\frac d2+1)}\right)\middle/\left(\frac{2\pi^{(d+1)/2}}{\Gamma(\frac d2+\frac12)}\right)\right.\geq \frac{\alpha^d}{3^{d}d^{d+1}\sqrt{\pi }}.\]

The probability that both $\langle g,e_o\rangle\geq 1-\frac{\alpha^2}{16d^2}$ and $\langle g,\Sigma^{1/2}e_{v_i}\rangle \leq -\alpha/2d$ for each $i$ is at least the probability that $\|g\|\geq 1$ and $\langle \tilde g,e_o\rangle\geq 1-\frac{\alpha^2}{16d^2}$. By rotational invariance of $g$ these are independent, so this probability is at least \[0.3 \cdot \frac{\alpha^d}{3^{d}d^{d+1}\sqrt{\pi }}\geq\frac{\alpha^d}{3^{d+2}d^{d+1}}.\]

\end{proof}

\begin{lemma} \label{lem:many-isolated-or-localized}
For any $d\geq 3$ $\delta>0$ and $0<\alpha\leq d$, there exists $N=N(d,\delta,\alpha)$ such that if $n>N$, then with probability at least $1-\delta$ with respect to $G(n,d)$, for any eigenvector $f$ with eigenvalue less than $-\alpha$ either
\begin{enumerate}
    \item $f$ has at least $\GWConst n/2$ singleton nodal domains, or
    \item There is a set of vertices $S\subset V$, $|S|\leq \delta n$ such that $\sum_{v\in S}f(v)^2\geq 1-\delta$.
\end{enumerate}
\end{lemma}
\begin{proof}

 Define $\mu=\mu(d,\lambda,\sigma)$ to be the distribution of the Gaussian wave $\sigma\cdot \Lambda_\lambda$ restricted to $B(o,1)$. Assume that $\tilde d(\mu,\nu_{G,f,1})\leq\epsilon$, for $\epsilon\leq \GWConst/2$ to be fixed later. We consider two cases depending on the relationship between $\sigma$ and $\epsilon$. If $\sigma$ is much larger than $\epsilon$, then the eigenvector is delocalized, and we can use \pref{lem:gauswave}. Otherwise, the eigenvector is localized.

 First, assume $\sigma\geq 10\epsilon d\alpha^{-1}$. Define $A$ to be the set of vectors $\overline x:=(x_o,x_{v_1},\ldots,x_{v_d})\in\R^{d+1}$ such that
 \begin{enumerate}
     \item 
     $\min\{|x_o|,|x_{v_1}|,\ldots,|x_{v_d}|\}\geq \frac{\sigma\alpha}{5d}$ and
     \item
     $x_o\cdot x_{v_i}<0$ for each $1\leq i\leq d$.
 \end{enumerate}

By \pref{lem:gauswave}, $\mu(A)\geq \GWConst$. By the definition of $A$, a given vector $\overline x\in A$ is such that all entries are of modulus at least $\frac{\sigma\alpha}{5d}$. Moreover, by the assumption on $\sigma$, we have $\epsilon\leq \frac{\sigma\alpha}{10d}$. Therefore, for a vector $\overline y:=(y_o,y_{v_1},\ldots,y_{v_d})$ such that $\|\overline x-\overline y\|\leq \epsilon$, the entries of $\overline y$ are of the same sign as the entries of $\overline x$. Therefore, if $x_o\cdot x_{v_i}<0$ for each $1\leq i\leq d$, then $y_o\cdot y_{v_i}<0$ for each $1\leq i\leq d$, meaning that if $B(o,1)$ is colored as per $\overline y$, then $\{o\}$ is a singleton nodal domain.

As $\tilde d(\mu,\nu_{G,f,1})\leq \epsilon$, we have $\nu_{G,f,1}(A_\epsilon)\geq \mu(A)-\epsilon\geq \GWConst/2$. By the previous paragraph, all vectors in $A_\epsilon$ correspond to singleton nodal domains, so there are at least $\GWConst n/2$ singleton domains of $f$ in $G$.

Now assume $\sigma<10\epsilon d\alpha^{-1}$. In this case, we will show that because $\nu_{G,f,1}$ is close to a Gaussian with low variance, the distribution of entries of $f$ must be concentrated around 0.

Denote by $\mu_0$ the distribution of the value on $o$ in $\mu$, and $\nu_0:=\nu_{G,f,0}$. Note that $\mu_0$ is the distribution $N(0,\sigma^2)$. The Euclidean distance between two points can only decrease when projecting onto a single coordinate, therefore the L\'evy Prokhorov distance can only decrease as well. This means that as $\tilde d(\mu,\nu_{G,f,1})\leq\epsilon$, then $\tilde d(\mu_0,\nu_0)\leq\epsilon$. Therefore for each $z\geq 0$,
\begin{equation*}\label{eq:prokpdf}
\Pr_{x\sim \nu_0}(x\in [-z-\epsilon,z+\epsilon])\geq \Pr_{x\sim \mu_0}(x\in[-z,z])-\epsilon.
\end{equation*}

Fix $z:=\sigma \sqrt{2\log \frac1\epsilon}$ and observe that by Gaussian tail bounds
\begin{equation}\label{eq:pro}
 \Pr_{x\sim \mu_0}(x\notin[-z,z])\leq 2\epsilon.
\end{equation} 
Also, by examining the endpoints of the interval, we have
\[
    \E_{x\sim\nu_{0}}\left(\textbf{1}\bigg[x\in [-z-\epsilon,z+\epsilon]\bigg]\cdot x^2\right)\leq \left(\sigma\sqrt{2\log\frac1\epsilon}+\epsilon\right)^2.
\]
By assumption $\sigma<10\epsilon d\alpha^{-1}$. Therefore 
\[
\left(\sigma\sqrt{2\log\frac1\epsilon}+\epsilon\right)^2\leq \left(10\epsilon d\alpha^{-1}\sqrt{2\log\frac1\epsilon}+\epsilon\right)^2=\left(1+10d\alpha^{-1}\sqrt{2\log\frac1\epsilon}\right)^2\epsilon^2\leq 250d^2\alpha^{-2}\epsilon^2\log\frac1\epsilon.
\]
As $\frac{1}{\epsilon}>\log \frac1\epsilon$ and $\E_{x\sim\nu_{0}}(x^2)=1$, this means that
\[\E_{x\sim\nu_{0}}\left(\textbf{1}\bigg[x\notin [-z-\epsilon,z+\epsilon]\bigg]\cdot x^2\right)\geq 1-250d^2\alpha^{-2}\epsilon.
\]
Combining this with \eqref{eq:pro} and the definition of $\nu_0$, this means that if $S=\{u\in V|f(u)^2\geq 2\sigma^2\log\frac1\epsilon\}$, then $|S|\leq 2\epsilon n$, and \[\sum_{u\in S}f(u)^2=\frac 1n\sum_{u\in S}nf(u)^2=\E_{x\sim\nu_{0}}\left(\textbf{1}\bigg[x\notin [-z-\epsilon,z+\epsilon]\bigg]\cdot x^2\right)\geq 1-250d^2\alpha^{-2}\epsilon.\]

It is therefore sufficient to choose $N$ as per \pref{thm:gaussianwave} for

\[
\epsilon<\min\left\{\frac{\GWConst}2,\frac{\alpha^{2}}{250d^2}\delta\right\}.
\]

\end{proof}

\section{Spectral radius bounds}\label{sec:specrad}
The main result of this section is \pref{lem:her-deg-specrad}, where we prove bounds on the spectral radius of high-girth graphs with bounded maximum degree and hereditary degree (defined below) approximately equal to $2$.
\begin{definition}
    The \emph{hereditary degree} of a graph $H$ is defined as:
    \[
        \max_{H'\subseteq H}\AvgDeg(H')
    \]
    where $\AvgDeg(H')=2|E(H')|/|V(H')|$.
\end{definition}

\begin{definition}
    Given a collection of edges $F$, we will use $v(F)$ to denote the number of vertices adjacent to $F$, and $c(F)$ to denote the number of connected components formed by $F$.
\end{definition}

\begin{definition}
    Given a graph $H$ and a collection of edges $F\subseteq E(H)$, we use $1_F$ to denote its indicator vector in $\R^{E(H)}$.  The \emph{spanning forest polytope} of $H$ is defined to be the convex hull of $\{1_F:F\text{ forest}\}$.
\end{definition}

We will also need the following two ingredients.
\begin{lemma}\cite{kesten1959symmetric}   \label{lem:forest-bound}
    If $T$ is a forest with maximum degree bounded by $\Delta$, then $\lambda_{\max}(A_T)\le2\sqrt{\Delta-1}$.
\end{lemma}

The following fact about the spanning forest polytope is a consequence of \cite[Theorem 13.21]{korte2011combinatorial}.
\begin{lemma}   \label{lem:sp-forest-polytope}
    The spanning forest polytope of a graph $H$ is equal to the feasible region of the following linear program:
    \begin{align*}
        y &\in\R^{E(H)} \\
        y&\ge 0\\
        \sum_{e\in F} y_e &\le v(F)-c(F) &\forall F\subseteq E(H).
    \end{align*}
\end{lemma}

\begin{lemma}   \label{lem:her-deg-specrad}
    Let $H$ be a graph with hereditary degree $2(1+\delta)$, maximum degree $\Delta$, and girth $g$.  Then:
    \[
        \lambda_{\max}(A_H) \le 2\frac{1+\delta}{1-\frac{1}{g}}\sqrt{\Delta-1}.
    \]
\end{lemma}
\begin{proof}
    Since $A_H$ is a symmetric matrix with nonnegative entries,
    \[
        \lambda_{\max}(A_H) = \max_{f\in\R^{V(H)}\setminus\{0\}} \frac{f^{\top}A_Hf}{\|f\|^2}.
    \]
    We will bound $f^{\top}A_Hf$ for any $f$.  Observe that:
    \[
        f^{\top}A_H f = \sum_{\{u,v\}\in E(H)} f_u f_v.
    \]
    We will prove that there is a spanning forest $T$ for which:
    \[
        \frac{1-\frac{1}{g}}{1+\delta} f^{\top}A_H f \le f^{\top}A_T f   \numberthis\label{eq:heavy-tree}
    \]
    which by \pref{lem:forest-bound} is bounded by $2\sqrt{\Delta-1}$ hence implying
    \[
        f^{\top}A_H f \le 2\frac{1+\delta}{1-\frac{1}{g}}\sqrt{\Delta-1}.
    \]
    To prove \pref{eq:heavy-tree} we exhibit a distribution $\calD$ on spanning forests such that:
    \[
        \E_{\bT\sim\calD}\left[f^{\top}A_{\bT}f\right] = \frac{1-\frac{1}{g}}{1+\delta} f^{\top}A_H f.
    \]
    Let $y\in\R^{E(H)}$ be the vector with $\frac{1-\frac{1}{g}}{1+\delta}$ in every entry.  We claim that $y$ is inside the spanning forest polytope of $H$.  To verify this, it suffices to check if $y$ satisfies the linear constraints given by the linear program description of the polytope from \pref{lem:sp-forest-polytope}.  By construction, each $y_e\ge 0$.

    For any $F\subseteq E(H)$, write it as $F_1\cup \dots \cup F_{c(F)}$ where each $F_i$ is a connected component given by $F$.  Since the girth of $H$ is at least $g$, for any $|F_i|<g$ we know $F_i$ forms a tree and hence $|F_i|=v(F_i)-1$.  For the remaining components, we know $|F_i|\le v(F_i)(1+\delta)$  by our bound on the hereditary average degree.  Now:
    \begin{align*}
        \sum_{e\in F} y_e &= \sum_{i=1}^{c(F)} \sum_{e\in F_i} y_e \\
        &= \sum_{i=1}^{c(F)} \frac{1-\frac{1}{g}}{1+\delta}|F_i| \\
        &= \sum_{i\in[c(F)]:|F_i|<g} \frac{1-\frac{1}{g}}{1+\delta}|F_i| + \sum_{i\in[c(F)]:|F_i|\ge g} \frac{1-\frac{1}{g}}{1+\delta}|F_i|\\
        &\le \sum_{i\in[c(F)]:|F_i|<g} (v(F_i)-1) + \sum_{i\in[c(F)]:|F_i|\ge g} \left(1-\frac{1}{g}\right)v(F_i)\\
        &\le \sum_{i=1}^{c(F)} (v(F_i)-1)\\
        &= v(F) - c(F).
    \end{align*}
    The second to last inequality follows from the fact that for a graph of girth $g$, a subgraph with at least $g$ edges has at least $g$ vertices. Since $y$ is in the spanning forest polytope of $H$ it must be expressible as a convex combination $p_1T_1+\dots+p_sT_s$ of indicator vectors of spanning forests in $H$.  Let $\calD$ be the distribution given by choosing spanning forest $T_i$ with probability $p_i$.  Notice that for $\bT\sim\calD$ the probability that any given edge $e$ is chosen is $\frac{1-\frac{1}{g}}{1+\delta}$.  Now:
    \begin{align*}
        \E_{\bT\sim\calD}\left[f^{\top}A_{\bT}f\right] &= \E_{\bT\sim\calD}\left[\sum_{\{u,v\} \in E(H)}\Ind[e\in\bT]f_uf_v\right] \\
        &= \sum_{\{u,v\}\in E(H)} f_u f_v \Pr[e\in\bT] \\
        &= \frac{1-\frac{1}{g}}{1+\delta}\sum_{\{u,v\}\in E(H)} f_u f_v \\
        &= \frac{1-\frac{1}{g}}{1+\delta} f^{\top}A_H f,
    \end{align*}
    which completes the proof.
\end{proof}

\section{$\ell_2$-localization implies many nodal domains}\label{sec:localized}
In this section $G$ is a $d$-regular graph and $f$ is a vector in $\R^{V(G)}$.  We prove that under some suitable assumptions on $G$ and $f$, it is not possible for $f$ to simultaneously be localized and have few nodal domains. 
Next, we verify that all of these conditions are simultaneously satisfied by random graphs and eigenvectors corresponding to sufficiently negative eigenvalues with high probability.

The conditions we impose on $G$ are:
\begin{enumerate}[wide, labelwidth=!, labelindent=0pt] 
    \item [{\bf Almost high-girth:}] 
    \mylabel{property:small-excess}{``almost-high girth''} There is $F\subseteq E(G)$ such that $|F|\le O(n^{1-c})$ and the girth of $G\setminus F$ is at least $c\log_{d-1}n$ for some absolute constant $c>0$.
    \item [{\bf Lossless edge expansion:}] \mylabel{property:lossless-edge-expansion}{``lossless edge expansion''} $\VExp_{\eps}(G) \ge d-2-\delta$ for some constants $\eps > 0$ and $0<\delta<d-2$.
\end{enumerate}

The conditions we impose on $f$ are:
\begin{enumerate}[wide, labelwidth=!, labelindent=0pt] 
    \item [{\bf $\ell_2$-localization:}] \mylabel{property:localization}{``$\ell_2$-localization''} There is a set $S\subseteq V(G)$ of size $\eps n$ such that $\|f_S\|^2\ge (1-\eta)\|f\|^2$ for some small constant $\eta > 0$ such that $4d\sqrt{\eta}<\delta\sqrt{d-2}$.
    \item [{\bf $\ell_{\infty}$-delocalization:}] \mylabel{property:eigvec-deloc}{{``$\ell_{\infty}$-delocalization''}} $\|f\|_{\infty}\le \displaystyle\frac{\log^{C}n}{\sqrt{n}}\|f\|$ for some constant $C$.
    \item [{\bf High energy:}] \mylabel{property:neg-qform}{``high energy''} $f^{\top}A_Gf = \lambda\|f\|^2$ for $\lambda < -2(1+2\delta)\sqrt{d-2}$.

\end{enumerate}
We note that the labels for the conditions on $G$ and $f$ are not definitions of those properties, but rather for readability in back-referencing.

The key result of this section is the following. We emphasize that all nodal domains considered are weak nodal domains of $f$ defined with respect to the graph $G$, and not its subgraphs.
\begin{lemma}   \label{lem:many-nodal-domains-if-localized} If $G$ and $f$ satisfy the above conditions then
    $f$ must have $\Omega\left(\frac{n}{\log^{2C+1}n}\right)$ singleton nodal domains.
\end{lemma}

A key lemma in service of proving \pref{lem:many-nodal-domains-if-localized} is:
\begin{lemma}   \label{lem:untempered-subgraph}
    Let $G,f$ and $S$ satisfy the above conditions, and let $c,d,\delta$ and $\eta$ be the parameters from above. If $f$ has fewer than $\frac{n}{\log^{2C+1}n}$ singleton nodal domains in $S$, then there is a subgraph $H$ of $G$ on vertex set $S$ such that:
    \begin{enumerate}
        \item \label{item:girth} The girth of $H$ is at least $c \log_{d-1} n$.
        \item \label{item:max-deg} The maximum degree of $H$ is at most $d-1$.
        \item \label{item:her-deg} The hereditary degree of $H$ is at most $2+\delta$.
        \item \label{item:quad-form} $f_S^{\top}A_H f_S\le (\lambda+4d\sqrt{\eta})\|f_S\|^2$.
    \end{enumerate}
\end{lemma}
\begin{proof}
    Let $H$ be the graph obtained by starting with $G[S]$, and then deleting the edge subgraph $$L \coloneqq L_+ \cup L_{\circ} \cup (F\cap E(G[S]))$$ where $L_+$ is the subgraph obtained by choosing every edge $\{u,v\}$ in $G[S]$ such that $f_S(u)f_S(v)\geq0$, and $L_{\circ}$ is obtained by choosing one arbitrary incident edge in $G[S]$ to each singleton nodal domain  $v\in S$ with degree $d$ in $G[S]$.
    \paragraph{Proof of \pref{item:girth}.} $H$ is a subgraph of $G\setminus F$ and hence has girth at least $c\log_{d-1}n$. 
    \paragraph{Proof of \pref{item:max-deg}.}  Every vertex $v$ with degree $d$ in $G[S]$ has an incident edge in $L$: indeed, if $v$ is a singleton nodal domain with degree $d$ in $G[S]$ then one of its incident edges is added to $L_\circ$; otherwise $v$ has a neighbor $u\in S$ such that $f_S(u)f_S(v)\geq0$, which means $\{u,v\}\in L_+$.  Consequently, every vertex in $H$ has degree bounded by $d-1$.
    \paragraph{Proof of \pref{item:her-deg}.}  For any $T\subseteq S$, since $|T|\le \eps n$, it must be the case that $|E(T,\ol{T})|\ge d-2-\delta$ by \ref{property:lossless-edge-expansion}.  Since $G$ is a $d$-regular graph, the average degree of $G[T]$ must be at most $2+\delta$.  Consequently since $H[T]$ is a subgraph of $G[T]$, the average degree of $H[T]$ is also bounded by $2+\delta$.
    \paragraph{Proof of \pref{item:quad-form}.}  First observe that:
    \begin{align*}
        \lambda\|f\|^2 &= f^{\top} A_G f \\
        &= f_S^{\top} A_G f_S + 2f_{\ol{S}}^{\top} A_G f_S + f_{\ol{S}}^{\top} A_G f_{\ol{S}} \\
        &\ge f_S^{\top} A_{G[S]} f_S - 2d\sqrt{\eta}\|f\|^2 - d\eta\|f\|^2 \\
        &\ge f_S^{\top} A_{G[S]} f_S - 3d\sqrt{\eta}\|f\|^2
    \end{align*}
    where the third line follows from \ref{property:localization}, Cauchy-Schwarz inequality, and $\lambda_{\max}(A_G)\le d$.
    Consequently $f_S^{\top}A_{G[S]}f_S \le (\lambda+3d\sqrt{\eta})\|f\|^2$.  Next, observe that:
    \begin{align*}
        f_S^{\top} A_{G[S]}f_S &= f_S^{\top} A_H f_S + f_S^{\top} A_L f_S \\
        &= f_S^{\top} A_H f_S + f_S^{\top} A_{L_{+}} f_S + f_S^{\top} A_{L_{\circ}} f_S + f_S^{\top} A_{F\cap E(G[S])} f_S \\
        &\ge f_S^{\top} A_H f_S + f_S^{\top} A_{L_{\circ}} f_S + f_S^{\top} A_{F\cap E(G[S])} f_S &(\text{since $f_S^{\top}A_{L_+}f_S\ge 0$}) \\
        &\ge f_S^{\top} A_H f_S - 2|L_{\circ}|\cdot\|f_S\|_{\infty}^2 - 2|F\cap E(G[S])|\cdot\|f_S\|_{\infty}^2 \\
        &\ge f_S^{\top} A_H f_S - \left(\frac{2}{\log n} + O(n^{-c}\log^{2C} n)\right)\|f\|^2,
    \end{align*}
    where the last inequality is because $|L_{\circ}|\le \frac{n}{\log^{2C+1} n}$ by assumption, $\|f_S\|_{\infty}^2 \le \frac{\log^{2C}n}{n}\|f\|^2$ by \ref{property:eigvec-deloc}, and $|F\cap E(G[S])| = O\left(n^{1-c}\right)$ by \ref{property:small-excess}.
    
    Chaining the above two inequalities together gives us:
    \[
        f_S^{\top} A_H f_S \le \left(\lambda + 3d\sqrt{\eta} + O\left(\frac{1}{\log n}\right)\right)\|f\|^2\le (\lambda + 4d\sqrt{\eta})\|f\|^2.
    \]
    Since $\lambda+4d\sqrt{\eta} < 0$ and $\|f\|^2\ge\|f_S\|^2$, the above is bounded by $(\lambda + 4d\sqrt{\eta})\|f_S\|^2$, completing the proof of \pref{item:quad-form}.
\end{proof}

We are now ready to prove \pref{lem:many-nodal-domains-if-localized}.
\begin{proof}[Proof of \pref{lem:many-nodal-domains-if-localized}]
    We prove the desired statement by contradiction.  If $f$ has less than $\frac{n}{\log^{2C+1} n}$ singleton nodal domains then consider the subgraph $H$ that is promised by \pref{lem:untempered-subgraph}.  On one hand by \pref{item:quad-form} of \pref{lem:untempered-subgraph}:
    \[
        f_S^{\top} A_H f_S \le (\lambda + 4d\sqrt{\eta})\|f_S\|^2 \le (-2(1+2\delta)\sqrt{d-2} + \delta\sqrt{d-2})\|f_S\|^2 = -2\left(1+\frac{3}{2}\delta\right)\sqrt{d-2}\|f_S\|^2.
    \]
    which implies that the spectral radius of $A_H$ is lower bounded by $2\left(1+\frac{3}{2}\delta\right)\sqrt{d-2}$.  On the other hand, by \pref{item:girth}, \pref{item:max-deg} and \pref{item:her-deg} of \pref{lem:untempered-subgraph} in conjunction with \pref{lem:her-deg-specrad} the spectral radius of $A_H$ is upper bounded by $\frac{(2+\delta)\sqrt{d-2}}{1-\frac{1}{c\log_{d-1}n}}$, which is at most $2(1+\delta)\sqrt{d-2}$, which is a contradiction.
\end{proof}

\begin{remark}[Sharpness of \pref{lem:untempered-subgraph} ]We remark that \pref{item:max-deg} in \pref{lem:untempered-subgraph} is the source of the $\lambda\le -2\sqrt{d-2}-\alpha$ hypothesis in \pref{thm:main}; reducing the degree of $H$ below $d-1$  would yield a larger spectral window in \pref{thm:main}. The entirely local argument of the Lemma is seen to be sharp by taking $G[S]=\cup_{i=1}^{k}T_i$ to be a disjoint union of finite $d-$ary trees $T_i$ of depth $O(\log\log n)$ such that the graph distance between any two trees is at least $2$, and $f$ to be an eigenfunction of a $(d-1)$-ary tree $T'_i\subset T_i$ with eigenvalue $\lambda\approx -2\sqrt{d-2}$ in each copy, and zero everywhere else. Then $f$ and $G$ satisfy the hypotheses of \pref{lem:untempered-subgraph} locally, $f$ has no singleton nodal domains in $G[S]$ (since each vertex has a path on which $f=0$ to the leaves of the tree), and there is no subgraph of $G[S]$ of maximum degree strictly less than $d-1$ satisfying \pref{item:quad-form}. Thus, improving \pref{lem:untempered-subgraph} will require either additional hypotheses or a more global examination of the structure of $G$ and $f$.
\end{remark}

\section{Many nodal domains in random regular graphs}\label{sec:many}
We are now ready to prove our main result.

\begin{proof}[Proof of \pref{thm:main}]
    By \pref{lem:many-isolated-or-localized}, with probability $1-o(1)$, every eigenvector $f$ either has $\Omega(n)$ singleton nodal domains or satisfies \ref{property:localization}.

    We define the following events:
    \begin{itemize}
        \item $\calE_1$: $G$ satisfies \ref{property:small-excess} with constant $\BikeConst$ and \ref{property:lossless-edge-expansion}; $f$ satisfies \ref{property:eigvec-deloc} with constant $\HYConst$ and \ref{property:neg-qform},
        \item $\calE_2$: $f$ has at least $\GWConst n/2$ singleton nodal domains,
        \item $\calE_3$: $f$ satisfies \ref{property:localization} and has fewer than $\GWConst n/2$ singleton nodal domains.
    \end{itemize}
    Clearly, when $\calE_2$ occurs there are ${\Omega}(n/\log^{2C+1} n)$ nodal domains.  Next, observe that when both $\calE_1$ and $\calE_3$ occur, the conditions of \pref{lem:many-nodal-domains-if-localized} are satisfied and, $f$ has ${\Omega}(n/\log^{2C+1}n)$ singleton nodal domains.

    Thus, it suffices to lower bound $\Pr[\calE_2\cup (\calE_1\cap\calE_3)]$.  Since $\calE_2$ and $\calE_3$ are mutually exclusive, $\calE_2$ and $\calE_3\cap\calE_1$ are also mutually exclusive, and hence:
    \begin{align}
        \Pr[\calE_2\cup (\calE_1\cap\calE_3)] = \Pr[\calE_2] + \Pr[\calE_1\cap\calE_3] \ge \Pr[\calE_2]+\Pr[\calE_3] - \Pr\left[\ol{\calE}_1\right]     \label{eq:many-nodal-domains-pr}
    \end{align}

    \pref{lem:many-isolated-or-localized} implies that $\Pr[\calE_2]+\Pr[\calE_3] = 1-o(1)$.  We further have $\Pr\left[\ol{\calE}_1\right]=o(1)$ by a combination of \pref{lem:excess-to-girth}, \pref{lem:lossless-edge-expansion} and \pref{thm:eigvec-deloc}.  Thus,
    \[
        \Pr[\calE_2\cup (\calE_1\cap\calE_3)] = 1-o(1),
    \]
    which completes the proof.
    
\end{proof}

\section{Large Nodal Domains in Expanders}\label{sec:xml}
In this section we prove that as a consequence of expansion in random graphs, for any eigenvector of a random $d$-regular graph, most vertices are part of a macroscopic nodal domain.  Key to our result in this section is the following lemma, which proves that by the expander mixing lemma, the only way to have the ``correct'' number of internal edges in a large subgraph is to have a large connected component.

\begin{lemma}   \label{lem:giant-component}
    Let $G$ be a $n$-vertex $d$-regular graph and let $S\subseteq V(G)$ of size $cn$, where $c$ is arbitrary. Also assume $\lambda(G)<d$.  Then $G[S]$ has a connected component of size at least:
    \[
        \left(c - \frac{2(1-c)\lambda(G)}{d-\lambda(G)}\right)\cdot n.
    \]
\end{lemma}
\begin{proof}
    By the expander mixing lemma (\pref{lem:exp-mix-lem}), we know that the average degree of $G[S]$ is:
    \begin{align*}
        \AvgDeg(G[S]) &= \frac{|E(S,S)|}{|S|} \\
        &\ge cd - \lambda(G)(1-c).
    \end{align*}
    Let the size of the connected component $C^*$ in $G[S]$ with maximum average degree be $c'n$.  We know that $\AvgDeg(G[C^*])$ is at least $\AvgDeg(G[S])$, and by the expander mixing lemma (\pref{lem:exp-mix-lem}):
    \begin{align*}
        \AvgDeg(G[C^*]) &= \frac{e(C^*, C^*)}{|C^*|} \\
        &\le c'd + \lambda(G)(1-c').
    \end{align*}
    Consequently, we have:
    \begin{align*}
        c'd + \lambda(G)(1-c') &\ge cd - \lambda(G)(1-c) \\
        c'(d-\lambda(G)) &\ge c(d+\lambda(G)) - 2\lambda(G)\\
        c' &\ge c\cdot\frac{d+\lambda(G)}{d-\lambda(G)} - \frac{2\lambda(G)}{d-\lambda(G)} \\
        &= c - \frac{2(1-c)\lambda(G)}{d-\lambda(G)}.
    \end{align*}
    which proves the claim.
\end{proof}

The result about nodal domains (which actually really applies to any signing of the vertices independent of being an eigenvector) in expanders is:
\begin{theorem}\label{thm:xml}
    Let $G$ be a $d$-regular graph and let $f$ be any eigenvector of $A_G$.  Suppose $C_1$ and $C_2$ be the two largest nodal domains in $f$, then $|C_1|+|C_2|\ge \left(1-\frac{2\lambda(G)}{d-\lambda(G)}\right)n$.
\end{theorem}

\begin{proof}
    Let $S_+\coloneqq \{v\in V(G): f(v)\ge 0\}$ and $S_- \coloneqq \{v\in V(G): f(v) < 0\}$.  Let's denote $|S_+|$ as $cn$ and $|S_-|$ as $(1-c)n$.  By \pref{lem:giant-component} we know that the largest component $C_+$ in $S_+$ has size at least $\left(c-\frac{2(1-c)\lambda(G)}{d-\lambda(G)}\right)\cdot n$ and the largest component $C_-$ in $S_-$ (which is distinct from $C_+$) has size at least $\left(1-c-\frac{2c\lambda(G)}{d-\lambda(G)}\right)\cdot n$.  It then follows:
    \begin{align*}
        |C_1|+|C_2| &\ge |C_+| + |C_-| \\
        &\ge \left(1-\frac{2\lambda(G)}{d-\lambda(G)}\right)\cdot n.
    \end{align*}
\end{proof}

\begin{remark}
    When $G$ is a random $d$-regular graph, then $\frac{2\lambda(G)}{d-\lambda(G)}=O\left(\frac{1}{\sqrt{d}}\right)$, and so for large enough $d$, the statement implies that a large constant fraction of the vertices are part of the two largest nodal domains.  For instance, when $d\ge 99$, at least half the vertices are part of the two largest nodal domains.
\end{remark}


\begin{thebibliography}{DGLS00}

\bibitem[AB11]{arora2011eigenvectors}
Sanjeev Arora and Aditya Bhaskara.
\newblock Eigenvectors of random graphs: delocalization and nodal domains.
\newblock {\em Preprint, available at http://www.cs.princeton.edu/\~
  bhaskara/files/deloc. pdf}, 2011.

\bibitem[Ber77]{berry1977regular}
Michael~V Berry.
\newblock Regular and irregular semiclassical wavefunctions.
\newblock {\em Journal of Physics A: Mathematical and General}, 10(12):2083,
  1977.

\bibitem[Ber08]{berkolaiko2008lower}
Gregory Berkolaiko.
\newblock A lower bound for nodal count on discrete and metric graphs.
\newblock {\em Communications in mathematical physics}, 278(3):803--819, 2008.

\bibitem[BGS02]{blum2002nodal}
Galya Blum, Sven Gnutzmann, and Uzy Smilansky.
\newblock Nodal domains statistics: A criterion for quantum chaos.
\newblock {\em Physical Review Letters}, 88(11):114101, 2002.

\bibitem[BHY19]{bauerschmidt2019local}
Roland Bauerschmidt, Jiaoyang Huang, and Horng-Tzer Yau.
\newblock {Local Kesten--McKay law for random regular graphs}.
\newblock {\em Communications in Mathematical Physics}, 369(2):523--636, 2019.

\bibitem[Bor19]{Bor}
Charles Bordenave.
\newblock A new proof of {F}riedman's second eigenvalue theorem and its
  extension to random lifts.
\newblock In {\em Annales scientifiques de l'Ecole normale sup{\'e}rieure},
  2019.

\bibitem[BOS07]{band2007nodal}
Ram Band, Idan Oren, and Uzy Smilansky.
\newblock Nodal domains on graphs-how to count them and why?
\newblock {\em arXiv preprint arXiv:0711.3416}, 2007.

\bibitem[BS19]{backhausz2019almost}
{\'A}gnes Backhausz and Bal{\'a}zs Szegedy.
\newblock On the almost eigenvectors of random regular graphs.
\newblock {\em Annals of Probability}, 47(3):1677--1725, 2019.

\bibitem[CH53]{courant2008methods}
Richard Courant and David Hilbert.
\newblock {\em Methods of mathematical physics: partial differential
  equations}.
\newblock John Wiley \& Sons, 1953.

\bibitem[DGLS00]{davies2000discrete}
E~Brian Davies, Graham Gladwell, Josef Leydold, and Peter~F Stadler.
\newblock Discrete nodal domain theorems.
\newblock {\em arXiv preprint math/0009120}, 2000.

\bibitem[DLL11]{dekel2011eigenvectors}
Yael Dekel, James~R Lee, and Nathan Linial.
\newblock Eigenvectors of random graphs: Nodal domains.
\newblock {\em Random Structures \& Algorithms}, 39(1):39--58, 2011.

\bibitem[Elo08]{elon2008eigenvectors}
Yehonatan Elon.
\newblock Eigenvectors of the discrete laplacian on regular graphs—a
  statistical approach.
\newblock {\em Journal of Physics A: Mathematical and Theoretical},
  41(43):435203, 2008.

\bibitem[Elo09]{elon2009gaussian}
Yehonatan Elon.
\newblock Gaussian waves on the regular tree.
\newblock {\em arXiv preprint arXiv:0907.5065}, 2009.

\bibitem[Fie75]{fiedler1975property}
Miroslav Fiedler.
\newblock A property of eigenvectors of nonnegative symmetric matrices and its
  application to graph theory.
\newblock {\em Czechoslovak mathematical journal}, 25(4):619--633, 1975.

\bibitem[Fri03]{friedman2003proof}
Joel Friedman.
\newblock A proof of {A}lon's second eigenvalue conjecture.
\newblock In {\em Proceedings of the thirty-fifth annual ACM symposium on
  Theory of computing}, pages 720--724, 2003.

\bibitem[HLW18]{HLW}
Shlomo Hoory, Nathan Linial, and Avi Wigderson.
\newblock Expander graphs and their applications.
\newblock {\em Bull. Amer. Math. Soc.}, 43(4):439–561, 2018.

\bibitem[HR20]{huang2020size}
Han Huang and Mark Rudelson.
\newblock Size of nodal domains of the eigenvectors of a graph.
\newblock {\em Random Structures \& Algorithms}, 57(2):393--438, 2020.

\bibitem[HY21]{huang2021spectrum}
Jiaoyang Huang and Horng-Tzer Yau.
\newblock Spectrum of random d-regular graphs up to the edge.
\newblock {\em arXiv preprint arXiv:2102.00963}, 2021.

\bibitem[Kes59]{kesten1959symmetric}
Harry Kesten.
\newblock Symmetric random walks on groups.
\newblock {\em Transactions of the American Mathematical Society},
  92(2):336--354, 1959.

\bibitem[KS97]{kottos1997quantum}
Tsampikos Kottos and Uzy Smilansky.
\newblock Quantum chaos on graphs.
\newblock {\em Physical review letters}, 79(24):4794, 1997.

\bibitem[KV12]{korte2011combinatorial}
Bernhard Korte and Jens Vygen.
\newblock {\em Combinatorial optimization}, volume~2.
\newblock Springer, 2012.

\bibitem[McK81]{mckay1981expected}
Brendan~D McKay.
\newblock The expected eigenvalue distribution of a large regular graph.
\newblock {\em Linear Algebra and its Applications}, 40:203--216, 1981.

\bibitem[Rud08]{rudnick2008quantum}
Ze’ev Rudnick.
\newblock Quantum chaos?
\newblock {\em Notices of the AMS}, 55(1):32--34, 2008.

\bibitem[Rud17]{rudelson2017delocalization}
Mark Rudelson.
\newblock Delocalization of eigenvectors of random matrices. lecture notes.
\newblock {\em arXiv preprint arXiv:1707.08461}, 2017.

\bibitem[Rud20]{rudelsonsimons}
Mark Rudelson.
\newblock Nodal domains of {G}(n,p) graphs.
\newblock \url{https://www.youtube.com/watch?v=ItMCuWiDRcI}, September 2020.

\bibitem[RV16]{rudelson2016no}
Mark Rudelson and Roman Vershynin.
\newblock No-gaps delocalization for general random matrices.
\newblock {\em Geometric and Functional Analysis}, 26(6):1716--1776, 2016.

\bibitem[Smi13]{smilansky2013discrete}
Uzy Smilansky.
\newblock Discrete graphs--a paradigm model for quantum chaos.
\newblock In {\em Chaos}, pages 97--124. Springer, 2013.

\bibitem[Zel17]{zelditch2017eigenfunctions}
Steve Zelditch.
\newblock {\em Eigenfunctions of the Laplacian on a Riemannian manifold},
  volume 125.
\newblock American Mathematical Soc., 2017.

\end{thebibliography}
\end{document}